\documentclass[12pt]{l4dc2023}

\title[Policy Evaluation in Distributional LQR]{Policy Evaluation in Distributional LQR}
\usepackage{times}
\usepackage{float}
\usepackage{algorithmic}
\usepackage{algorithm}

\author{%
 \Name{Zifan Wang}$^1$ \Email{zifanw@kth.se} \\
 \Name{Yulong Gao}$^2$ \Email{yulong.gao@cs.ox.ac.uk} \\
 \Name{Siyi Wang}$^3$ \Email{siyi.wang@tum.de} \\
 \Name{Michael M. Zavlanos}$^4$ \Email{michael.zavlanos@duke.edu} \\
  \Name{Alessandro Abate}$^2$ \Email{alessandro.abate@cs.ox.ac.uk}\\
  \Name{Karl H. Johansson}$^1$ \Email{kallej@kth.se} 
  \\
  \addr $^1$ Division of Decision and Control Systems, KTH Royal Institute of Technology, Sweden\\
   \addr $^2$ Department of Computer Science, University of Oxford, UK\\
   \addr $^3$ Chair of Information-oriented Control, Technical University of Munich, Germany\\
    \addr $^4$ Department of Mechanical Engineering and Materials Science, Duke University, USA
}

\begin{document}

\maketitle

\begin{abstract}%
Distributional reinforcement learning (DRL) enhances the understanding of the effects of the randomness  in the environment by letting agents learn the distribution of a random return, rather than its expected value as in standard RL. 
At the same time, a main challenge in DRL is that policy evaluation in DRL typically relies on the representation of the return distribution, which needs to be carefully designed.
In this paper, we address this challenge for a special class of DRL problems that rely on discounted linear quadratic regulator (LQR) for control, advocating for a new distributional approach to LQR, which we call \emph{distributional LQR}. Specifically, we provide  a closed-form expression of the distribution of the random return which, remarkably, is applicable to all exogenous disturbances on the dynamics, as long as they are independent and identically distributed (i.i.d.).
While the proposed exact return distribution consists of infinitely many random variables, we show that this distribution can be approximated by a finite number of random variables, and the associated approximation error can be analytically bounded under mild assumptions. 
Using the approximate return distribution, we propose a zeroth-order policy gradient algorithm for risk-averse LQR using the Conditional Value at Risk (CVaR) as a measure of risk. 
Numerical experiments are provided to illustrate our theoretical results.
\end{abstract}




\section{Introduction}

In reinforcement learning, the value of implementing a policy at a given state is captured by a value function, which models the expected sum of returns following this prescribed policy. 
Recently, \citet{bellemare2017distributional}  proposed the notion of distributional reinforcement learning (DRL), which learns the return distribution of a policy from a given state, instead of only its expected return. 
Compared to the scalar expected value function, the return distribution is infinite-dimensional and contains far more information. 
It is, therefore, not surprising that a few DRL algorithms, including C51 \citep{bellemare2017distributional}, D4PG \citep{barth2018distributed}, QR-DQN \citep{dabney2018distributional} and SDPG \citep{singh2022sample}, dramatically improve the empirical performance in practical applications over their non-distributional counterpart. 

In DRL, the practical effectiveness of algorithms builds on the theory by \citet{bellemare2017distributional}, where the distributional Bellman operator is shown to be a contraction in the (maximum form of) the Wasserstein metric between probability distributions.
However, it is usually difficult to  characterise the exact return distribution in DRL with finite data. 
%
Approximations of the return distribution are thus necessary to make it computable in practice.
To address this challenge, \citet{bellemare2017distributional} propose a categorical method that partitions the return distribution into a finite number of uniformly spaced atoms in a fixed region.
One drawback of this method is that it relies on prior knowledge of the range of the returned values. 
To address this limitation, a quantile function method 
\citep{dabney2018distributional} and a sample-based method \citep{singh2022sample} have been recently proposed. 
%
However,  these works cannot provide an analytical expression for the approximation error, and computational cost needs to be decided manually to guarantee approximation accuracy.

In this paper, we characterise the return distribution of the random cost for the classical discounted linear quadratic regulator (LQR) problem, which we term \emph{distributional LQR}.  
To the best of our knowledge, the return distribution in LQR has not been explored in the literature. 
Our contributions are summarised as follows:

\begin{enumerate}
    \item 
    We provide an analytical expression of the random return for distributional LQR problems and prove that this return function is a fixed-point solution of the distributional Bellman equation. 
    Specifically, we show that the proposed analytical expression consists of infinitely many random variables and holds for arbitrary i.i.d. exogenous disturbances, e.g., non-Gaussian noise or noise with non-zero mean.  
    \item We develop an approximation of the distribution of the random return using a finite number of random variables. Under mild assumptions, we theoretically show that the sup of the difference between the exact  and  approximated  return distributions deceases linearly with the numbers of random variables: this is also validated by numerical experiments.
    \item The proposed analytical return distribution provides a theoretical foundation for distributional LQR, allowing for general optimality criteria for policy improvement. In this work, we employ the return distribution to analyse risk-averse LQR problems using the Conditional Value at Risk (CVaR) as the risk measure. Since the gradient of CVaR is generally difficult to compute analytically, we propose a risk-averse policy gradient algorithm that relies on the zeroth-order optimisation to seek an optimal risk-averse policy. Numerical experiments are provided to showcase this application.   
\end{enumerate}
%

\noindent \textbf{Related Work:} 
Most closely related to the problem considered in this paper is work on reinforcement learning for LQR, which focuses on learning the expected return through interaction with the environment; see, e.g., \citet{dean2020sample,tu2018least,fazel2018global,malik2019derivative,li2021distributed,yaghmaie2022linear,zheng2021sample}.
For example, \citet{fazel2018global} propose a model-free policy gradient algorithm for LQR and showed its global convergence with finite polynomial computational and sample complexity.
Moreover, \citet{zheng2021sample} study model-based reinforcement learning for the Linear Quadratic Gaussian problems, in which a model is first learnt from data and then used to design the policy. 
However, all these works rely on the expected return instead of the return distribution, 
hence these methods cannot  be applied here.

Since the return distribution captures the intrinsic randomness of the long-term cost, it provides a natural framework to consider more general optimality criteria, e.g., optimal risk-averse policies. 
There exist recent works on risk-averse policy design for DRL, including
\citet{singh2020improving,dabney2018implicit,tang2019worst}. For example, the work in \citet{dabney2018implicit} use the quantile function to approximate the return distribution, which is then applied to design risk-sensitive policies for Atari games. On the other hand, \citet{singh2020improving} show that risk-averse DRL achieves robustness against system disturbances in continuous control tasks. 
All these works focus on empirical improvements in specific tasks, however, without theoretical analysis.
Related to this paper is also work on risk-sensitive LQR, which has been studied in 
\citet{van2015distributionally,tsiamis2021linear,kim2021distributional,chapman2021toward,kishida2022risk}. Similarly, these methods however do not analyse the return distribution. 

\vspace{-0.2cm}
\section{Problem Statement}
Consider a discrete-time linear dynamical system:
\begin{align}
    x_{t+1} = A x_t + B u_t +v_t ,
\end{align}
where $x_t \in \mathbb{R}^n$, $u_t\in \mathbb{R}^p$, $v_t \in \mathbb{R}^n$ are the system state, control input, and the exogenous disturbance, respectively. 
We assume that the exogenous disturbances $v_t$ with bounded moments, $t\in \mathbb{N}$, are i.i.d. sampled from a distribution $ \mathcal{D}$ of arbitrary form.

\vspace{-0.2cm}
\subsection{Classical LQR}
The canonical  LQR problem aims to find a control policy $\pi: \mathbb{R}^n \rightarrow \mathbb{R}^p$ to minimise the objective
\begin{align}
    J(u) 
    = \mathbb{E}\left[ \sum_{t=0}^{\infty} \gamma^t (x_t^T Q x_t + u_t^T R u_t)  \right],
\end{align}
where $Q,R$ are positive-definite constant matrices and $\gamma \in (0,1)$ is a discount parameter. Given a control policy $\pi$, let $V^{\pi}(x) = \mathbb{E}\left[ \sum_{t=0}^{\infty} \gamma^k (x_t^T Q x_t + u_t^T R u_t)  \right]$ denote the expected return from an initial state $x_0 =x$ with $ u_t = \pi( x_t)$.
For the static linear policy $\pi(x_t)=K x_t$, the value function $V^{\pi}(x)$ satisfies the Bellman equation
\begin{align}\label{eq:Bellman expectation}
    V^{\pi}(x) 
    &= x^T (Q + K^T R K) x + \gamma \mathop{\mathbb{E}}_{X' = (A+BK)x+v_0 } [V^{\pi}(X')],  
\end{align} 
where the capital letter $X'$ denotes a random variable over which we take the expectation.

When the exogenous disturbances $v_t$ are normally distributed with zero mean, the value function is known to take the quadratic form  $V^{\pi}(x) = x^T P x +q$, where $P>0$ is the solution of the Lyapunov equation $P = Q+ K^T R K + \gamma A_K^T P A_K$ and $q$ is a scalar related to the variance of $v_t$. In particular, the optimal control feedback gain is obtained as $K^*=-\gamma(R+\gamma B^TPB)^{-1}PA$ and $P$ is the solution to the classic Riccati equation $P = \gamma A^T P A - \gamma^2 A^T P B (R+\gamma B^T P B)^{-1} B^T P A+Q$.




\subsection{Distributional LQR}
Motivated by the advantages of DRL in better understanding the effects of the randomness in the environment and in considering more general optimality criteria, in this paper we propose a distributional approach to the LQR problem.
Unlike classical reinforcement learning, which relies on expected returns, DRL \citep{bdr2022} relies on the distribution of random returns.
The return distribution characterises the probability distribution of different returns generated by a given policy and, as such, it contains much richer information on the performance of a given policy compared to the expected return.
In the context of LQR, we
 denote by $G^{\pi}(x)$ the random return using the static control strategy $u_t = \pi( x_t)$ from the initial state $x_0=x$, which is defined as
\begin{align}
    G^{\pi}(x) 
    = \sum_{t=0}^{\infty} \gamma^t (x_t^T Q x_t + u_t^T R u_t) , \quad  u_t = \pi( x_t),x_0 = x.
\end{align}
It is straightforward to see that the expectation of $G^{\pi}(x)$ is equivalent to the value function $V^{\pi}(x)$. 
The standard Bellman equation in \eqref{eq:Bellman expectation} decomposes the long-term expected return into an immediate stage cost plus the expected return of future actions starting at the next step. 
Similarly, we can define the distributional Bellman equation for the random return as 
\begin{align}\label{eq:rv:Bellman}
    G^{\pi}(x) & \mathop{=}^{D} x^TQx+\pi(x)^TR\pi(x) + \gamma G^{\pi}(X'), \quad X' = Ax+B\pi(x)+v_0.
\end{align}
Here we use the notation $\mathop{=}\limits^{D}$ to denote that two random variables $Z_1,Z_2$ are equal in distribution, i.e., $Z_1 \mathop{=}\limits^{D} Z_2$. Note  that $X'$ denotes a random variable, as in \eqref{eq:Bellman expectation}.
Compared to the expected return in LQR, which is a scalar, here the return distribution is infinite-dimensional and can have a complex form. 
It is  challenging to estimate an infinite-dimensional function exactly with finite data and thus an approximation of the return distribution is  necessary in practice. 

In this paper,  we first analytically characterise the random return for the LQR problem. Then we  show how to  approximate the  distribution of the random return using finite random variables, so such that the approximated distribution is computationally tractable and the  approximation error is bounded. The proposed distributional LQR framework allows us to consider more general optimality criteria, which we demonstrate by using the proposed return distribution to develop a policy gradient algorithm for risk-averse LQR.

\section{Main Results}
\subsection{Exact Form of the Return Distribution}
In this section, we precisely characterise the distribution of the random return  that satisfies the distributional  Bellman equation \eqref{eq:rv:Bellman}. 
Given a static linear policy $\pi(x_t)=K x_t$, we denote by $G^K(x)$ the random return $G^{\pi}(x)$  under the policy $\pi(x_t)$ from the initial state $x_0=x$ , which is defined as
\begin{align*}
    G^{K}(x) =  \sum_{t=0}^{\infty} \gamma^t x_t^T (Q+K^T R K) x_t  , \quad x_0 = x.
\end{align*}
The random return  $G^{K}(x)$ satisfies the following distributional Bellman equation  
\begin{align}\label{eq:rv:bellman}
    G^{K}(x) & \mathop{=}^{D} x^TQ_Kx + \gamma G^{K}(X'),\quad X'=A_K x+ v_0,
\end{align}
where $A_K:=A+BK$ and $Q_K := Q+ K^T R K$. 
In the following theorem, we provide an explicit expression of the random return $G^K(x)$.

\begin{theorem}
Suppose that the feedback gain  $K$ is stabilizing, i.e., $A_K=A+BK$ is stable. Let  \begin{align}\label{eq:dist_func}
    G^{K}(x) = x^T P x +  \sum_{k = 0 }^{\infty}  \gamma^{k+1}  w_k^T P w_k +  2 \sum_{k = 0 }^{\infty} \gamma^{k+1} w_k^T P A_K^{k+1}x + 
 2  \sum_{k = 1 }^{\infty} \gamma^{k+1} w_k^T P \sum_{\tau=0}^{k-1} A_K^{k-\tau}w_{\tau},
\end{align}
where  $P$ is obtained from the algebraic Riccati equation $P = Q+ K^T R K + \gamma A_K^T P A_K$, and the random variables $w_k \sim \mathcal{D} $ are independent from each other for all $k\in\mathbb{N}$. Then, the random variable $G^{K}(x)$ defined in 
    \eqref{eq:dist_func} is a fixed point solution to the distributional Bellman equation \eqref{eq:rv:bellman}.
\end{theorem}
\begin{proof}
Recall that $X' = A_K x + v_0$, where $v_0$ is a  random variable sampled from the distribution $\mathcal{D}$ and is independent from $w_k$, $k\in \mathbb{N}$, in \eqref{eq:dist_func}. 
Substituting \eqref{eq:dist_func} into the right hand side of the equation \eqref{eq:rv:bellman}, we have that
\begin{align*}
    & x^T(Q+K^T R K)x + \gamma G^{K}(X') \\
    =  & x^TQ_Kx + \gamma X'^T P X' +  \sum_{t=0}^{\infty} \gamma^{t+2} w_t^T P w_t + 2 \sum_{t=0}^{\infty} \gamma^{t+2} w_t^T  P A_K^{t+1}  X' \\
    &+ 2  \sum_{t=1}^{\infty} \gamma^{t+2} w_t^T P A_K \sum_{i=0}^{t-1} A_K^{t-1-i} w_i  \\
    = & x^TQ_Kx + \gamma (A_K x + v_0)^T P (A_K x + v_0) + \gamma^2 \sum_{t=0}^{\infty} \gamma^t w_t^T P w_t  + 2 \gamma^2 \sum_{t=1}^{\infty} \gamma^t w_t^T P \sum_{i=0}^{t-1} A_K^{t-i} w_i \\
    &+ 2\gamma^2 \sum_{t=0}^{\infty} \gamma^t w_t^T  P A_K^{t+1}  (A_K x + v_0) \\
    = & x^T(Q_K + \gamma A_K^T P A_K)x + \underbrace{\gamma v_0^T P v_0 + \gamma^2 \sum_{t=0}^{\infty} \gamma^t w_t^T P w_t }_{:=T_1} +\underbrace{ 2\gamma v_0^T PA_K x + 2\gamma^2 \sum_{t=0}^{\infty} \gamma^t w_t^T  P A_K^{t+2} x}_{:=T_2} \\
    & + \underbrace{2 \gamma^2 \sum_{t=1}^{\infty} \gamma^t w_t^T P  \sum_{i=0}^{t-1} A_K^{t-i} w_i + 2\gamma^2 \sum_{t=0}^{\infty} \gamma^t w_t^T  P  A_K^{t+1} v_0}_{:=T_3}.
\end{align*}
Define $ \xi_0 :=v_0$, $\xi_t = w_{t-1}$, $t=1,2,\ldots$. 
From the definition of the term $T_1$, we have that
\begin{align*}
    T_1 & = \gamma v_0^T P v_0 + \gamma^2 \sum_{t=0}^{\infty} \gamma^t w_t^T P w_t 
     \mathop{=}^{k=t+1}  \gamma \xi_0^T P \xi_0 + \gamma \sum_{k=1}^{\infty} \gamma^k \xi_{k}^T P \xi_{k} 
    =  \gamma \sum_{k=0}^{\infty} \gamma^k \xi_{k}^T P \xi_{k}.
\end{align*}
For the term $T_2$, we have that
\begin{align*}
    T_2 & = 2\gamma v_0^T PA_K x + 2\gamma^2 \sum_{t=0}^{\infty} \gamma^t w_t^T  P A_K^{t+2} x 
    = 2\gamma \xi_0^T P A_K x + 2\gamma^2 \sum_{t=0}^{\infty} \gamma^t \xi_{t+1}^T  P A_K^{t+2} x \\
    &\mathop{=}^{k=t+1}  2\gamma \xi_0^T P A_K x + 2\gamma \sum_{k=1}^{\infty} \gamma^k \xi_{k}^T  P A_K^{k+1} x 
    =  2\gamma  \sum_{k=0}^{\infty} \gamma^k \xi_k^T P A_K^{k+1} x.
\end{align*}
Using similar techniques for the term $T_3$, we obtain that
$T_3 = 2 \gamma \sum_{k=1}^{\infty} \gamma^{k} \xi_{k}^T P A_K  \sum_{i=0}^{k-1} A_K^{k-1-i} \xi_i.$ 
Due to the fact that $P =Q+ K^T R K + \gamma A_K^T P A_K$, we have 
\begin{align}\label{eq:proof:P1:temp1}
    & x^TQ_Kx + \gamma G^{K}(X') 
    =  x^T P x + T_1 +T_2 +T_3 \nonumber \\
    = & x^T P x + \gamma \sum_{k=0}^{\infty} \gamma^k \xi_{k}^T P \xi_{k} + 2\gamma  \sum_{k=0}^{\infty} \gamma^k x^T P A_K^{k+1} \xi_k + 2 \gamma \sum_{k=1}^{\infty} \gamma^{k} \xi_{k}^T P A_K  \sum_{i=0}^{k-1} A_K^{k-1-i} \xi_i,
\end{align}
which is in the same form as in \eqref{eq:dist_func}.
Since $\{ \xi_k\}_{k=0}^{\infty}$ and $\{ w_k\}_{k=0}^{\infty}$ are i.i.d.,
we have that the two random variables \eqref{eq:dist_func} and \eqref{eq:proof:P1:temp1} have the same distribution, i.e.,
$G^{K}(x)  \mathop{=}\limits^{D}  x^TQ_Kx + \gamma G^{K}(X').$
\end{proof}


\subsection{Approximation of the Return Distribution with Finite Parameters}\label{Sec:Approxreturn}

In this section, we show how to approximate the random return defined in \eqref{eq:dist_func} using a finite number of random variables. Considering only the first $N$ terms in the summations in the expression in \eqref{eq:dist_func}  and disregarding the terms for $k$ larger than $N$ yields the following:   
\begin{align}\label{eq:approxreturn}
    {G}^{K}_N(x) =  x^T P x +  \sum_{k = 0 }^{N-1}  \gamma^{k+1}  w_k^T P w_k +  2 \sum_{k = 0 }^{N-1} \gamma^{k+1} w_k^T P A_K^{k+1}x + 
 2 \sum_{k = 1 }^{N-1} \gamma^{k+1} w_k^T P \sum_{\tau=0}^{k-1} A_K^{k-\tau}w_{\tau}.
\end{align}
Let $F^{K}_x$ and ${F}^{K}_{x,N}$ denote the cumulative distribution function (CDF) of $G^{K}(x)$ and ${G}_N^{K}(x)$, respectively. The following theorem provides an upper bound on the difference between $F^{K}_x$ and ${F}^{K}_{x,N}$, and shows that  the sequence $\{{G}^{K}_N(x)\}_{N\in \mathbb{N}}$ converges pointwise in distribution to $G^{K}(x)$,  $\forall x\in\mathbb{R}^n$. 

\begin{theorem}\label{Theorem:approx}
Assume that the probability density functions of $w_k$ exist and are bounded, and satisfy $\mathbb{E}[w_k^T w_k] \leq \sigma_0^2$, $\mathbb{E}[\left\|w_k\right\|_2] \leq \mu_0 $, for $\forall k\in \mathbb{N}$.
Suppose that the feedback gain  $K$ is stabilizing such that $\left\|A_K\right\|_2 = \rho_K <1$. Then, the sup difference between the CDFs $F^K_x$ and ${F}^{K}_{x,N}$ is bounded by
\begin{align}\label{eq:approximat:bound}
     \sup_{z}|{F}^{K}_x(z)-{F}^{K}_{x,N}(z) | \leq C \gamma^N,
\end{align}
where $C$ is a constant that depends on the matrices $A,B,Q,R,K$, the initial state value $x$, and the parameters $\gamma, \rho_K, \sigma_0,\mu_0$. 
\end{theorem}

\begin{proof}
Define $Y_N:= G^K(x) - G^K_N(x)$, we have
\begin{align}\label{eq:upp_bound_temp1}
     &\sup_{z}|{F}^{K}_x(z)-{F}^{K}_{x,N}(z) | 
     =\sup_{z}|\mathbb{P}(G^K_N(x) \leq z) -\mathbb{P}(G^K(x)\leq z) | \nonumber \\
     =& \sup_{z}|\mathbb{P}(G^K_N(x)\leq z) -\mathbb{P}(G^K_N(x) +Y_N\leq z) | \nonumber \\
     =& \sup_{z}\Big|\mathbb{P}(G^K_N(x)\leq z) \int_{-\infty}^{\infty} \mathbb{P}(Y_N = t)dt -\int_{-\infty}^{\infty} \mathbb{P}(G^K_N(x) \leq z-t) \mathbb{P}(Y_N = t) dt \Big| \nonumber \\
     =& \sup_{z}\Big| \int_{-\infty}^{\infty} \mathbb{P}(Y_N = t) \big(  F^K_{x,N}(z) -F^K_{x,N}(z-t) \big) dt \Big|.
\end{align}
Since the random variables $w_t$ are i.i.d for all $t>0$ and the probability density function of $w_t$ exists, the function $F^K_{x,N}$ is continuous and differentiable. 
Applying the mean value theorem, when $t>0$ there exists a point $z'\in[z-t,z]$ such that $F^K_{x,N}(z) -F^K_{x,N}(z-t) = f^K_{x,N}(z') t$, where $f^K_{x,N}$ is the probability density function of $G^K_N(x)$. Since the probability density function of $w_t$ is bounded, it further follows  that $f^K_{x,N}$ is bounded.   Then, we have that $|F^K_{x,N}(z) -F^K_{x,N}(z-t)| = |f^K_{x,N}(z') t| \leq L_0 |t|$, where $L_0$ is an upper bound of the probability function $f^K_{x,N}$. Following a similar argument, we can show that this inequality holds when $t\leq 0$. Substituting this inequality into \eqref{eq:upp_bound_temp1}, we obtain
\begin{align}\label{eq:upp_bound_temp2}
   \sup_{z}|{F}^{K}_x(z)-{F}^{K}_{x,N}(z) | \leq \sup_{z}\Big| \int_{-\infty}^{\infty} \mathbb{P}(Y_N = t) L_0 |t| dt \Big| 
     = L_0 \mathbb{E}|Y_N|.
\end{align}
From the definition of $Y_N$, we obtain that
\begin{align*}
     Y_N =& 
     \sum_{k = N }^{\infty}  \gamma^{k+1}  w_k^T P w_k +  2 \sum_{k = N }^{\infty} \gamma^{k+1} w_k^T P A_K^{k+1}x + 2 \sum_{k = N }^{\infty} \gamma^{k+1} w_k^T P \sum_{\tau=0}^{k-1} A_K^{k-\tau}w_{\tau} \nonumber \\
     \mathop{=}^{t=k-N}&    \gamma^N \Big(  \sum_{t = 0 }^{\infty} \gamma^{t+1}  w_{t+N}^T P w_{t+N} +  2 \sum_{t = 0 }^{\infty} \gamma^{t+1} w_{t+N}^T P A_K^{t+N+1}x \nonumber \\
     &+ 2 \sum_{t = 0 }^{\infty} \gamma^{t+1} w_{t+N}^T P \sum_{\tau=0}^{t+N-1} A_K^{t+N-\tau}w_{\tau}  \Big).
\end{align*}
Taking the expectation of the absolute value of $Y_N$, we have
\begin{align*}
    \mathbb{E}|Y_N| \leq &\gamma^N \Big( \sum_{t = 0 }^{\infty} \gamma^{t+1} \mathbb{E} | w_{t+N}^T P w_{t+N} |   +  2 \sum_{t = 0 }^{\infty} \gamma^{t+1} \mathbb{E}| w_{t+N}^T P A_K^{t+N+1}x|\nonumber \\
    &+ 2 \sum_{t = 0 }^{\infty} \gamma^{t+1} \mathbb{E} |w_{t+N}^T P \sum_{\tau=0}^{t+N-1} A_K^{t+N-\tau}w_{\tau} |  \Big) .
\end{align*}
We handle the terms in the above inequality one by one. For the first term, we have that
\begin{align}\label{eq:upp_bound_temp4}
    \sum_{t = 0 }^{\infty} \gamma^{t+1} \mathbb{E} | w_{t+N}^T P w_{t+N} | \leq \sum_{t = 0 }^{\infty} \gamma^{t+1}  \mathbb{E}| \lambda_{\max}(P) w_{t+N}^T w_{t+N} | \leq \lambda_{\max}(P) \sigma_0^2 \frac{\gamma}{1-\gamma}.
\end{align}
For the second term, we have that
\begin{align}\label{eq:upp_bound_temp5}
    &2 \sum_{t = 0 }^{\infty} \gamma^{t+1} \mathbb{E}| w_{t+N}^T P A_K^{t+N+1}x|\leq  2\mu  \sum_{t = 0 }^{\infty} \gamma^{t+1} \left\| P \right\|_2 \left\| A_K^{t+N+1} \right\|_2 \left\|x \right\|_2 \nonumber \\
    \leq& 2\mu  \sum_{t = 0 }^{\infty} \gamma^{t+1} \left\| P\right\|_2 \rho_K^{t+N-1} \left\|x \right\|_2 \leq 
    2\mu \left\| P\right\|_2 |x|\frac{\gamma \rho_K^{N-1}}{1-\gamma \rho_K} \leq 2\mu \left\| P\right\|_2 |x|\frac{\gamma}{1-\gamma \rho_K},
\end{align}
where the second inequality is due to the fact that $\left\|A_K^{t+N+1} \right\|_2\leq (\left\|A_K \right\|_2)^{t+N+1} \leq \rho_K^{t+N+1}$ and the last inequality follows from the fact that $N\geq 1$.
For the third term, we have that
\begin{align}\label{eq:upp_bound_temp8}
    &2 \sum_{t = 0 }^{\infty} \gamma^{t+1} \mathbb{E} |w_{t+N}^T P \sum_{\tau=0}^{t+N-1} A_K^{t+N-\tau}w_{\tau} | 
    \leq   2 \sum_{t = 0 }^{\infty} \gamma^{t+1} \mathbb{E} \left[ \left\| w_{t+N}^T\right\|_2  \left\| P \right\|_2  \left\|\sum_{\tau=0}^{t+N-1} A_K^{t+N-\tau} w_{\tau}\right\|_2 \right] \nonumber \\
    \leq &  2 \mu \left\| P \right\|_2  \sum_{t = 0 }^{\infty} \gamma^{t+1}   \mathbb{E} \left[ \left\|\sum_{\tau=0}^{t+N-1} A_K^{t+N-\tau} w_{\tau}\right\|_2 \right]
    \leq  2 \mu \left\| P \right\|_2  \sum_{t = 0 }^{\infty} \gamma^{t+1}   \mathbb{E} \left[ \sum_{\tau=0}^{t+N-1} \left\|A_K^{t+N-\tau} \right\|_2\left\|w_{\tau}\right\|_2 \right] \nonumber\\
    \leq  &2 \mu^2 \left\| P \right\|_2 \sum_{t = 0 }^{\infty} \gamma^{t+1}  \sum_{\tau=0}^{t+N-1}  \rho_K^{t+N-\tau} \leq 2 \mu^2 \left\| P \right\|_2 \sum_{t = 0 }^{\infty} \gamma^{t+1} \frac{\rho_K}{1-\rho_K} \leq 2 \mu^2 \left\| P \right\|_2 \frac{\gamma \rho_K}{(1-\gamma)(1-\rho_K)},
\end{align}
where the second inequality is due to the fact that $w_{\tau}$ and $w_{t+N}$ are independent and the  second to last inequality follows from the fact that $\sum_{\tau=0}^{t+N-1}  \rho_K^{t+N-\tau} = \sum_{\tau=1}^{t+N} \rho_K^{\tau} \leq \frac{\rho_K}{1-\rho_K} $.
Combining \eqref{eq:upp_bound_temp4}, \eqref{eq:upp_bound_temp5} and \eqref{eq:upp_bound_temp8}, we have that
\begin{align}
    &\sup_{z}|{F}^{K}_x(z)-{F}^{K}_{x,N}(z) | \leq  L_0 \mathbb{E}|Y_N| \nonumber \\
    \leq & L_0 \gamma^N \Big(\lambda_{\max}(P) \sigma_0^2 \frac{\gamma}{1-\gamma} + 2\mu \left\| P\right\|_2 |x|\frac{\gamma}{1-\gamma \rho_K}
    +  2 \mu^2 \left\| P \right\|_2 \frac{\gamma \rho_K}{(1-\gamma)(1-\rho_K)} \Big) 
    :=  C \gamma^N . \nonumber 
\end{align}
The proof is complete and also yields the expression of the constant $C$. 
\end{proof}

\begin{remark}
The bound on the distribution approximation in \eqref{eq:approximat:bound} relies on the conditions of Theorem~\ref{Theorem:approx}, which ensure that the PDF of $G^K_N$ is continuous and bounded. Note that these conditions are not particularly strict, and indeed hold for many  noise distributions commonly used in linear dynamical systems, including Gaussian and uniform. Future work will investigate relaxations of these conditions. 
\end{remark}

\subsection{Numerical Experiments on Quality of the Approximation of the Return Distribution}\label{Sec:NemVer}

In the following experiment, we consider a scalar model with matrices $A=B=1$. Similarly, the weighting matrices in the LQR cost are chosen as $Q=R=1$. The exogenous disturbances are standard normal distributions with zero mean. 

Even for this scalar system, it is impossible to simplify the expression of the exact return distribution, which still depends on an infinite number of random variables.  
Thus, as a baseline for the return distribution, we generate an empirical distribution that approximates the true distribution of the random return. More specifically, we use the Monte Carlo (MC) method to obtain 10000 samples of the random return and use the sample frequency over evenly-divided regions as an approximation of the probability density function. According to the law of large numbers, the empirical distribution approaches the real one as the  number of trials increases. 
Note that, although the MC method provides an alternative way to approximate the return distribution, it relies on using sufficiently many samples that can be time-consuming, and its (statistical) approximation error is generally difficult to analyse. Thus, the MC method is not applicable for practical policy evaluation of distributional LQR, and in this experiment, it is used only to verify our approximate return distribution.
In comparison, the approximate return distribution using finite number of random variables in this paper is analytical for policy evaluation and the corresponding approximation error can be bounded: as such, it is further usable for policy optimisation, as shown in Section~\ref{Sec:riskaversecontrol}. We denote here by $f_N$ the distribution of the approximated random return $G^K_N(x_0)$ obtained considering $N$ random variables. 

\begin{figure}[t]
    \centering
	\subfigure[$\gamma=0.6$, $x_0=1$.]{
	\includegraphics[scale=0.24]{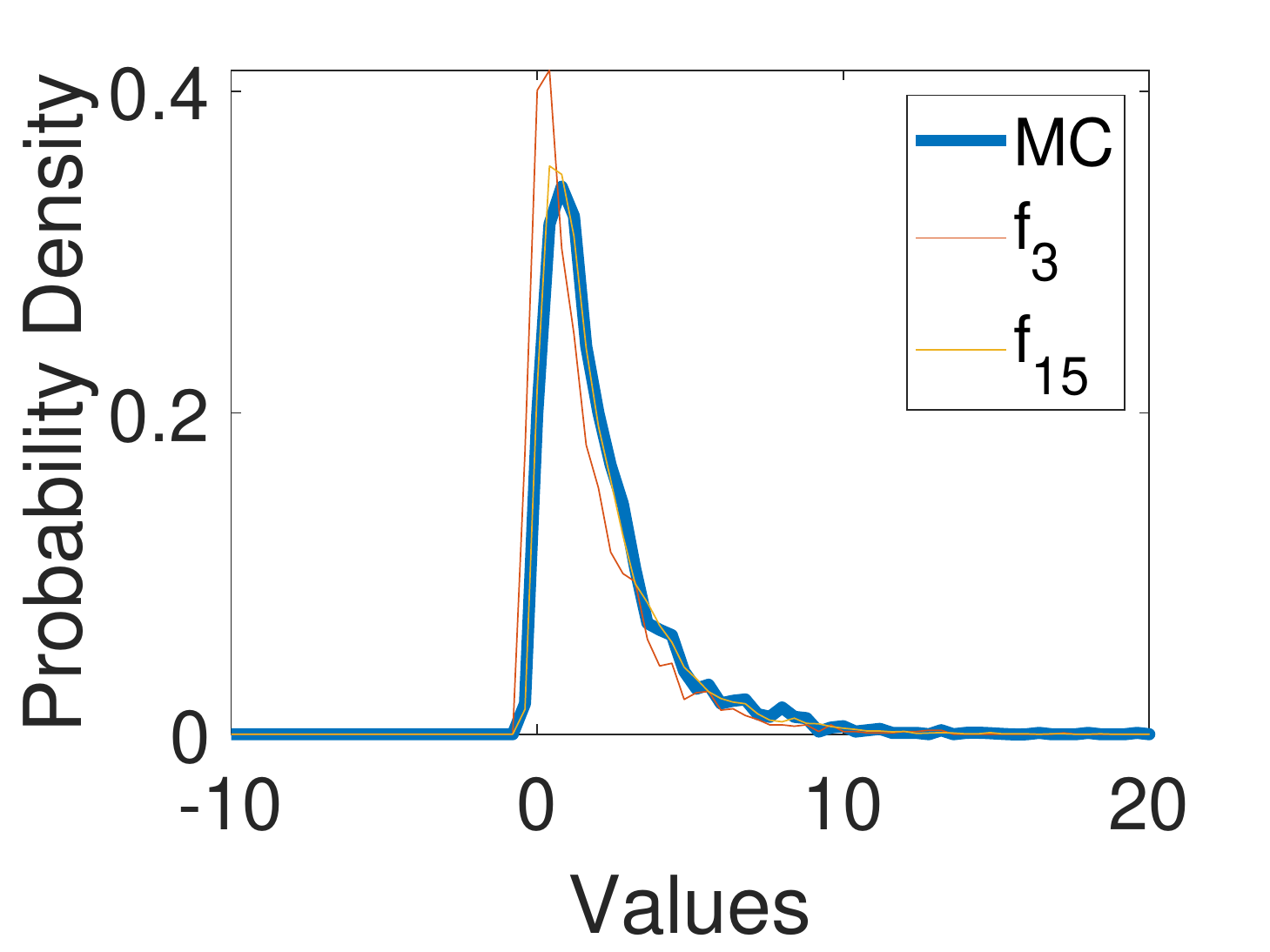}}
	\subfigure[$\gamma=0.8$, $x_0=1$.]{
	\includegraphics[scale=0.24]{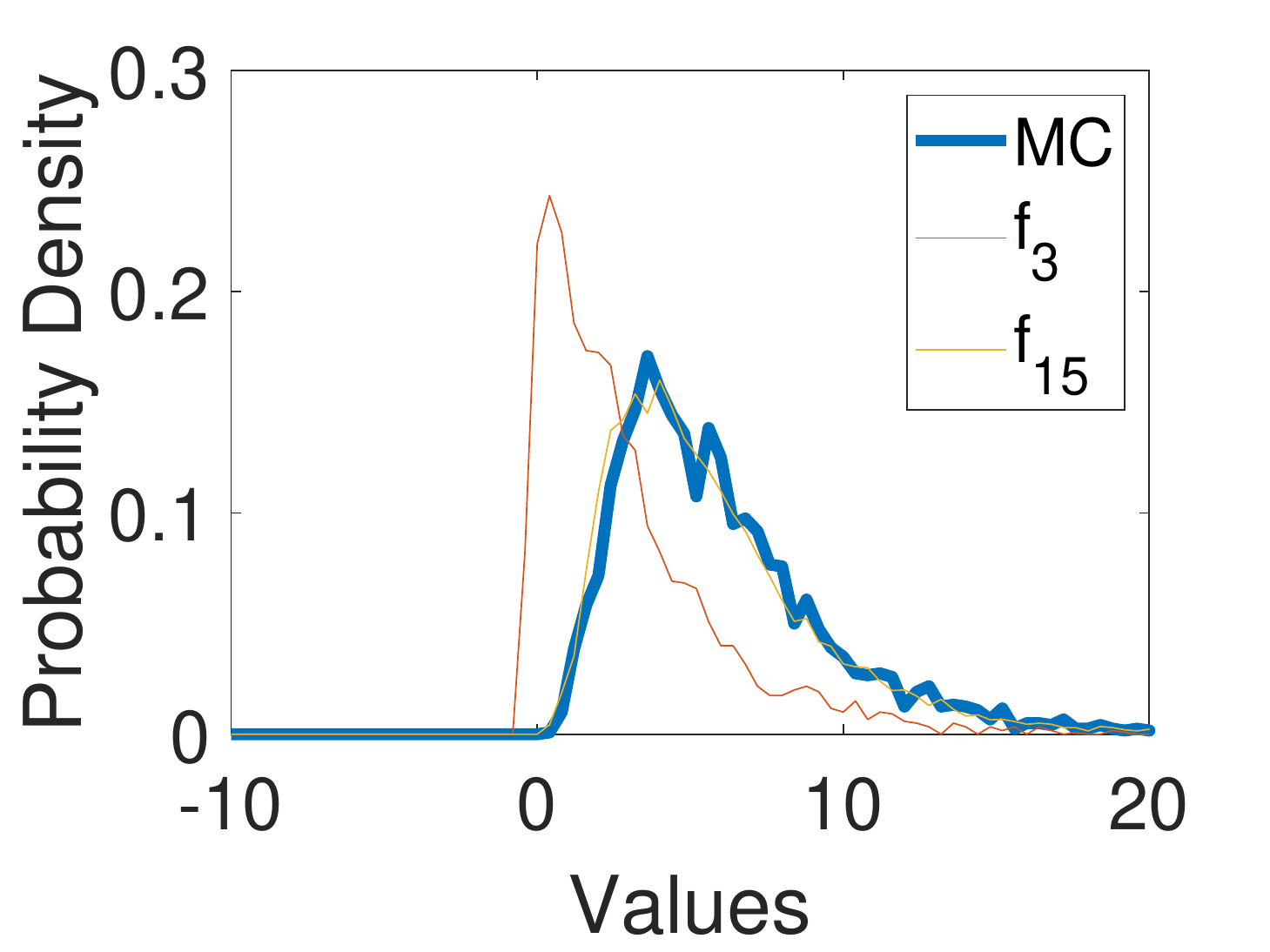}}
        \subfigure[$\gamma=0.6$, $x_0=8$.]{
	\includegraphics[scale=0.24]{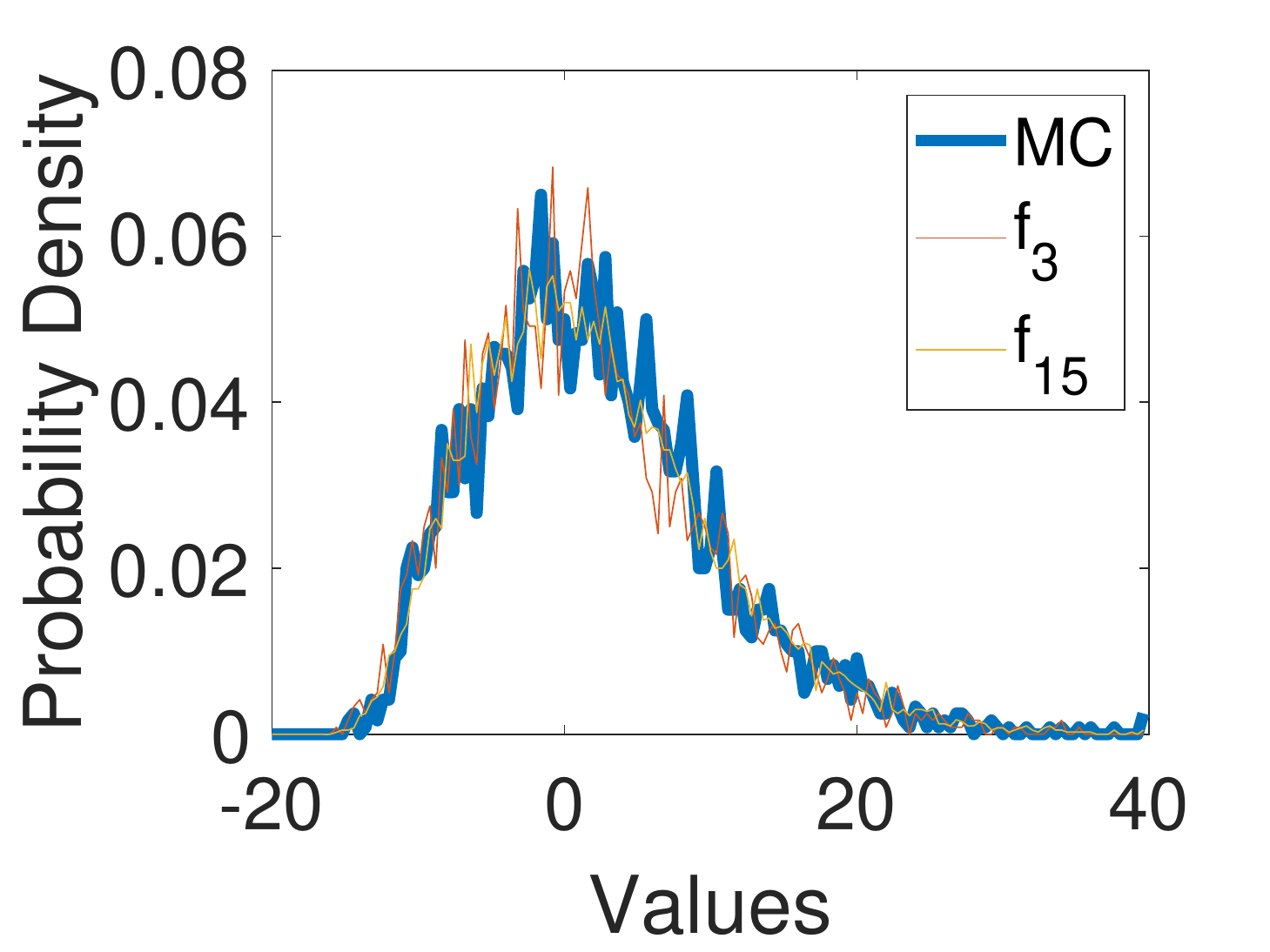}}
        \subfigure[$\gamma=0.85$, $x_0=8$.]{
	\includegraphics[scale=0.24]{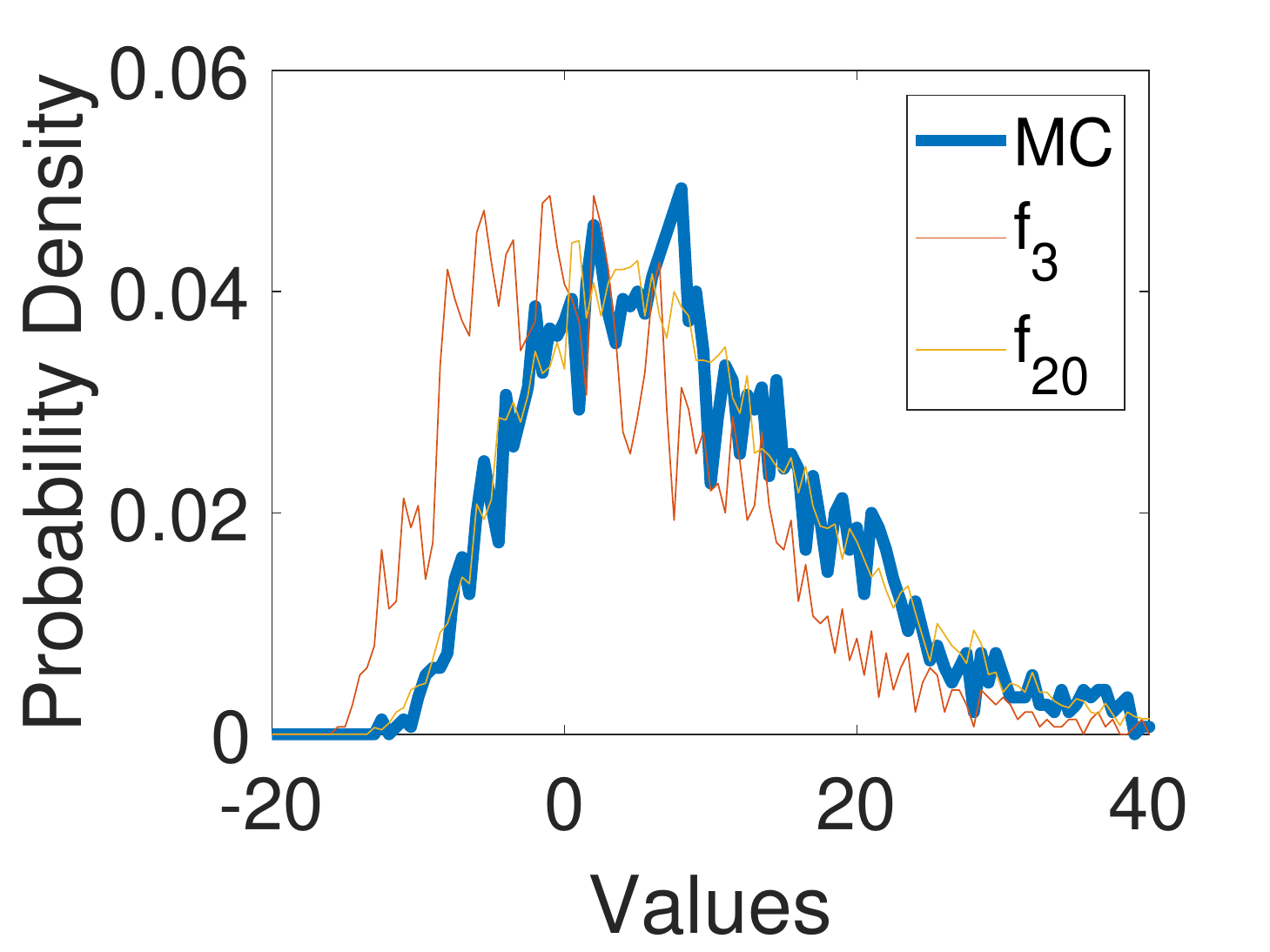}}
	\caption{Return distribution and its approximation with finite number of random variables for different $\gamma$ and $x_0$. MC denotes the distribution returned by the Monte Carlo method and $f_N$ denotes the distribution of the approximated random return $G^K_N(x_0)$.}\label{fig:S1}
 \vspace{-0.2in}
\end{figure}

We fix the feedback gain as $K= -0.4684$ and select different values of $\gamma$ and $x_0$. The results are shown in Fig.~\ref{fig:S1}. Specifically,  Fig.~\ref{fig:S1} (a) and (c) show that when $\gamma$ is small, the return distribution can be well approximated using only few random variables ($N=3$ works well). 
However, when $\gamma$ approaches $1$, more random variables are needed for an accurate approximation: we employ $N=15$ and $N=20$ random variables in the case of $\gamma=0.8$ and $\gamma=0.85$, respectively, as shown in Fig.~\ref{fig:S1} (b) and (d). Moreover, the value of the initial state $x_0$ has an influence on the shape of the return distribution, which can be clearly observed from the scalar case. When $x_0$ is large, the random variable $w_k^T P A_K^{k+1}x_0$ dominates and, therefore, its distribution is close to a Gaussian distribution, as shown in Fig.~\ref{fig:S1} (c) and (d). If instead $x_0$ is  small, then the random variable $w_k^T P w_k$ plays a leading role, so the overall distribution is close to the chi-square one, as shown in Fig.~\ref{fig:S1} (a) and (b). 
In conclusion,  when $N$ is large, the approximate distribution is closer to the distribution obtained from the MC method, and thus to the true distribution.

\section{Application to Risk-Averse LQR}\label{Sec:riskaversecontrol}
In this section, we consider a risk-averse LQR problem and leverage the closed-form expression of the random return $G^K(x)$ to obtain an optimal policy. Since the distribution of the random return $G^K(x)$ consists of an infinite number of random variables, it is computationally unwieldy. 
Instead, we employ the approximate random return $G^K_N(x)$ proposed in Section~\ref{Sec:Approxreturn}.
As a risk measure for the problem at hand, 
we select the well-known Conditional Value at Risk (CVaR) \citep{rockafellar2000optimization}. 
We then construct an approximate risk-averse objective function, as $\hat{\mathcal{C}}_N(K):={\rm{CVaR}}_{\alpha}\left[ {G}_N^K(x)\right]$.  
For a random variable $Z$ with the CDF $F$ and a risk level $\alpha \in (0,1]$, the ${\rm{CVaR}}$ value is defined as ${\rm{CVaR}}_{\alpha}[Z] = \mathbb{E}_F[Z| Z>Z^{\alpha}]$, where $Z^{\alpha}$ is the $1-\alpha$ quantile of the distribution of the random variable $Z$. 
Given this objective function, the goal is to find the optimal risk-averse controller, that is, to select the feedback gain $K$ that minimises $\hat{\mathcal{C}}_N(K)$. 

\subsection{Risk-Averse Policy Gradient Algorithm}


In what follows, we propose a policy gradient method to solve this problem. We assume that the matrices $A,B,Q,R$ are known. The first-order gradient descent step is hard to compute as it hinges on the gradient of the CVaR function. Therefore, we rely on zeroth-order optimisation to derive the policy gradient, as detailed in Algorithm~\ref{alg:algorithm_PG}. 



\begin{algorithm}[t]
\caption{Risk-Averse Policy Gradient}
\begin{algorithmic}[1]
\REQUIRE  initial values $K_0$, $x$, step size $\eta$, smoothing parameter $\delta$, and dimension $n$
    \FOR {$episode \; t=1,\ldots,T$}
        \STATE Sample $\hat{K}_t = K_t + U_t$, where $U_t$ is drawn at random over matrices whose norm is $\delta$;
        \STATE Compute the distribution of the random variable ${G}_N^{\hat{K}_t}$;
        \STATE Compute $\hat{\mathcal{C}}_N(\hat{K}_t)$;
        \STATE $K_{t+1}= K_t - \eta g_t$, where $g_t= \frac{n}{\delta^2} \Big(\hat{\mathcal{C}}(\hat{K}_t)- \hat{\mathcal{C}}(\hat{K}_{t-1}) \Big) U_t $.
    \ENDFOR
\end{algorithmic}\label{alg:algorithm_PG}
\end{algorithm}

Specifically, at each episode $t$, we sample an approximate feedback gain $\hat{K}_t = K_t +U_t$, where $U_t$ is drawn uniformly at random from the set of matrices with norm $\delta$. Given $\hat{K}_t$, we compute the approximate distribution of the random return ${G}_N^{\hat{K}_t}(x)$ in \eqref{eq:approxreturn} and the value of $\hat{\mathcal{C}}_N(\hat{K}_t)$. Then, we can perform the feedback gain update as
  $K_{t+1}= K_t - \eta g_t$,
where $g_t= \frac{n}{\delta^2} \Big(\hat{\mathcal{C}}(\hat{K}_t)- \hat{\mathcal{C}}(\hat{K}_{t-1}) \Big) U_i $. Here, the zeroth-order residual feedback technique proposed in \citet{zhang2022new} is used to reduce the variance.
The theoretical analysis of this algorithm is left as our future work.

\begin{figure}[t]
    \centering
	\subfigure[The $K$ values  when $\alpha=1$.]{
	\includegraphics[scale=0.24]{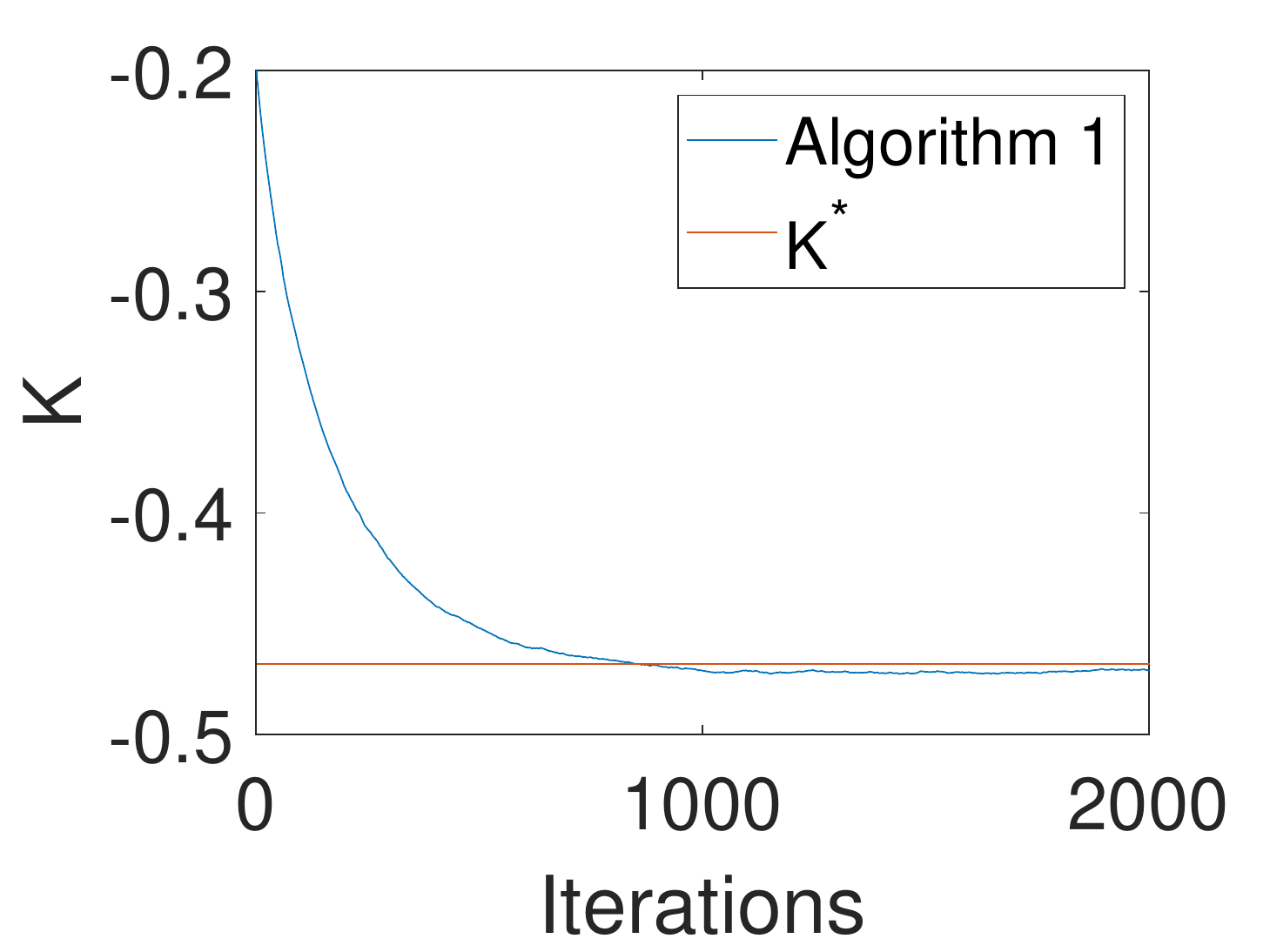}}
	\subfigure[The ${\rm{CVaR}}$ values  when $\alpha=1$.]{
	\includegraphics[scale=0.24]{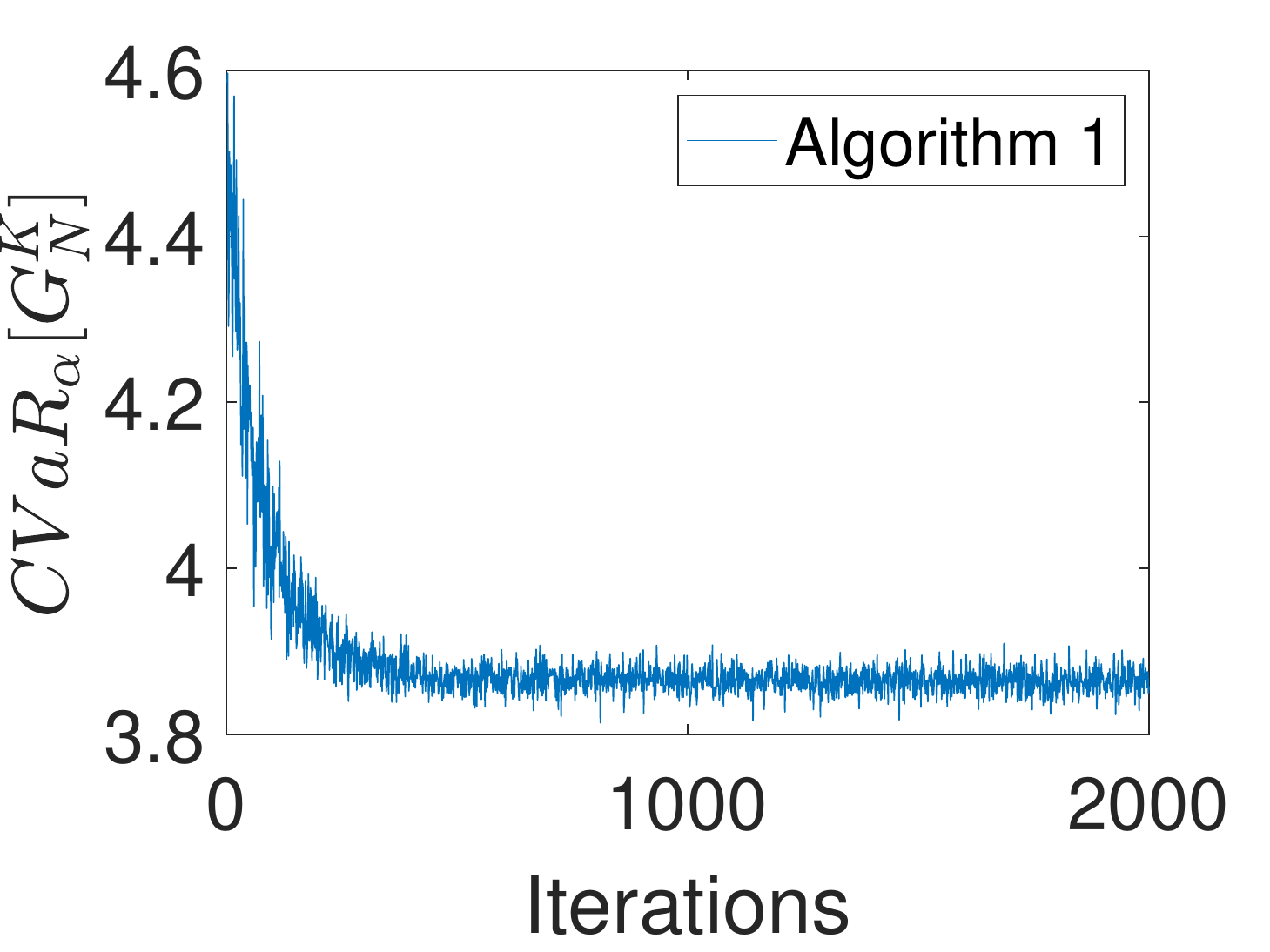}}
        \subfigure[The $K$ values  when $\alpha=0.4$.]{
	\includegraphics[scale=0.24]{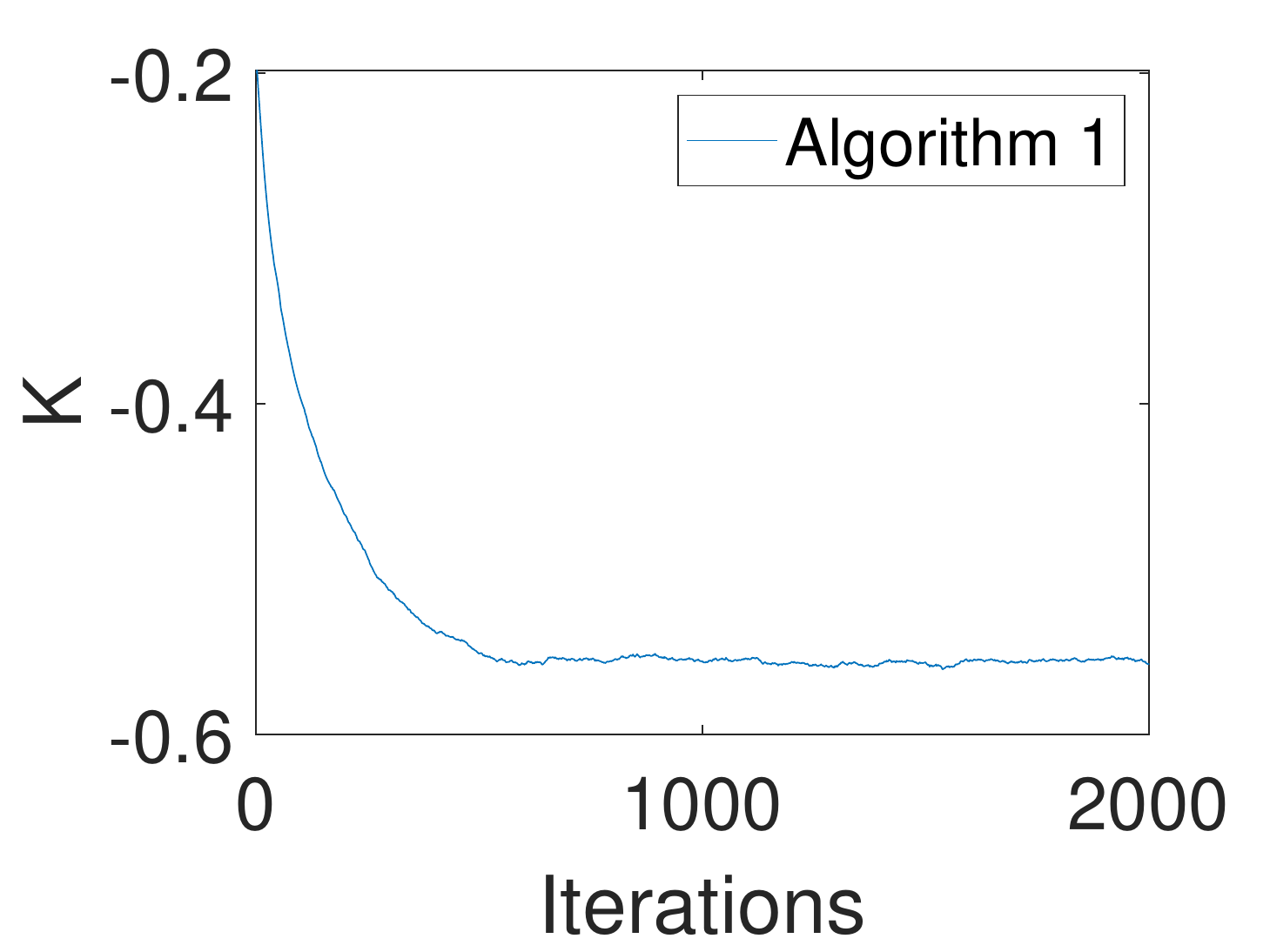}}
        \subfigure[The ${\rm{CVaR}}$ values when $\alpha=0.4$.]{
	\includegraphics[scale=0.24]{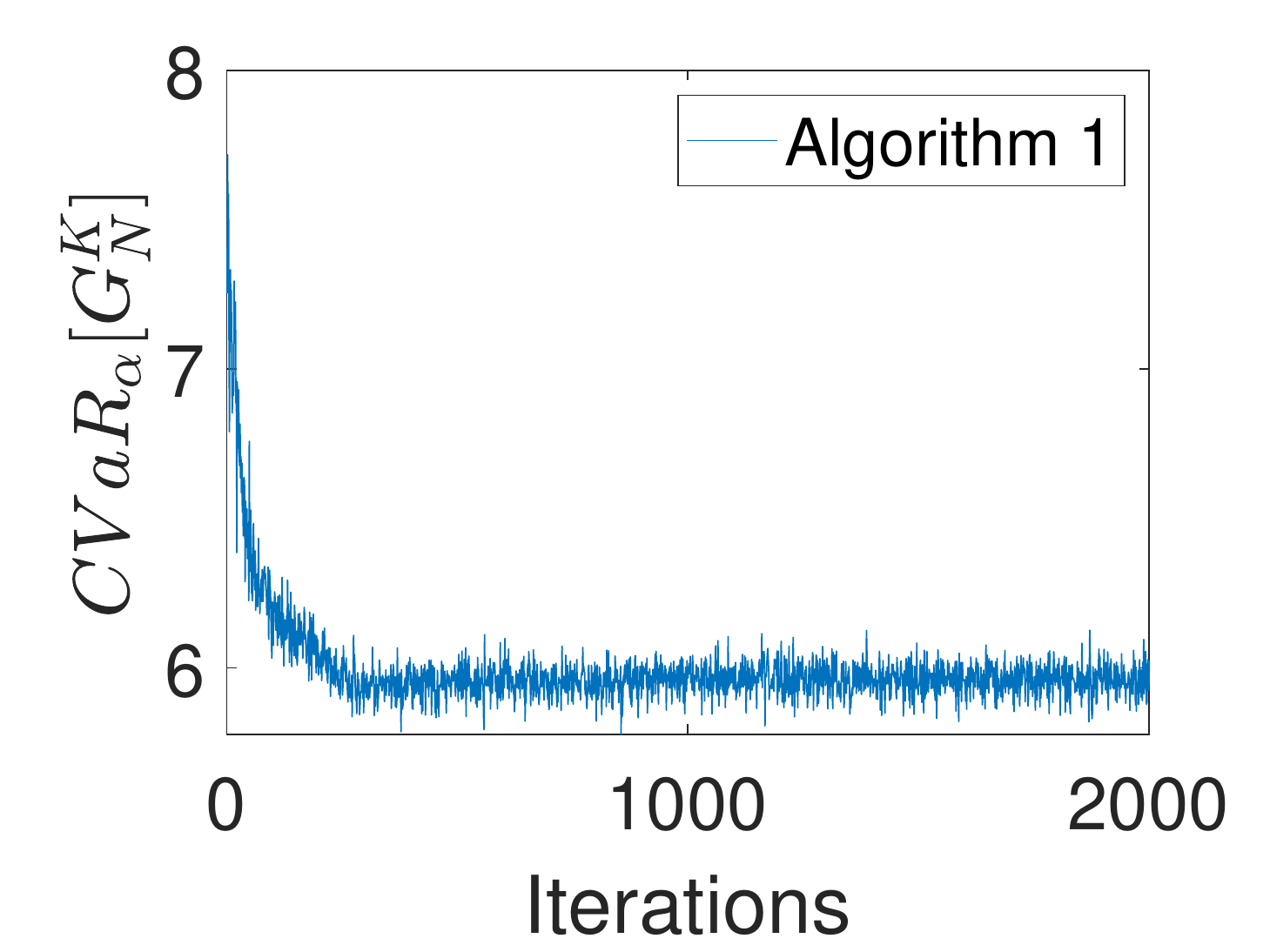}}
	\caption{Risk-averse control  using Algorithm~\ref{alg:algorithm_PG}. The solid lines are averages over 20 runs.} \label{fig:S2}
\end{figure}

\subsection{Numerical Experiments}
Next, we consider a risk-averse LQR  problem and experimentally illustrate the performance of Algorithm~\ref{alg:algorithm_PG}. We illustrate our approach for the same  scalar system with the same cost function as in Section~\ref{Sec:NemVer}. 
The other parameters are selected as $\gamma=0.6$, $\delta=0.1$, $\eta=0.0004$, $N= 10$, respectively. The initial controller is set as $K_0=-0.2$, which is a stable one.

We first set $\alpha=1$: in this case, the risk-averse control problem is reduced to a risk-neutral control problem. Therefore, we can use traditional LQR techniques to compute the optimal feedback gain $K^{*}=-0.4684$. 
We run the proposed risk-averse policy gradient Algorithm~\ref{alg:algorithm_PG} and the simulation results are presented in Fig.~\ref{fig:S2} (a) and (b). Specifically, in Fig.~\ref{fig:S2} (a), the controller $K$ returned by Algorithm 1 converges to $K^{*}$, which verifies our proposed method for the risk-neutral case. Fig.~\ref{fig:S2} (b) illustrates the values of ${\rm{CVaR}}$ achieved by Algorithm~\ref{alg:algorithm_PG}.
Additionally, we select $\alpha=0.4$ to find the optimal risk-averse controller. The simulation results are presented in Fig.~\ref{fig:S2} (c) and (d). We see that $K$  converges to $-0.55$, which leads to a smaller $A+BK$ compared to $K^{*}=-0.4684$.



\section{Conclusions} 
We have proposed a new distributional approach to the classic discounted LQR problem. Specifically, we first provided an analytic expression for the exact random return  that depends on infinitely many random variables. 
Since the computation of this expression is difficult in practice, we also proposed an approximate expression for the distribution of the random return that only depends on a finite number of random variables, and have further characterised the error between these two distributions. Finally, we utilised the proposed random return to  obtain an optimal controller for a risk-averse LQR problem using the CVaR as a measure of risk. 
To the best of our knowledge, this is a first framework for distributional LQR: it inherits the advantages of DRL methods compared to standard RL methods that rely on the expected return to evaluate the effect of a given policy, but it also provides an analytic expression for the return distribution, an area where current DRL methods significantly lack. 
Future research includes analyzing the theoretical convergence of risk-averse policy gradient algorithms and exploring a model-free setup where the system matrices are unknown.

\acks{This work is supported in part by the Knut and Alice Wallenberg Foundation, the Swedish Strategic Research Foundation, the Swedish Research Council,  AFOSR under award \#FA9550-19-1-0169, and  NSF under award CNS-1932011.}


\end{document}